\documentclass[reqno]{amsart}

\usepackage{amsmath,amssymb,amsthm,amsfonts,enumerate, enumitem,hyperref,cleveref}
\usepackage{dsfont}
\usepackage[all]{xy}
\usepackage[T1]{fontenc}
\usepackage{tikz}

\DeclareMathOperator{\Aut}{Aut}
\DeclareMathOperator{\inn}{Inn}
\DeclareMathOperator{\laut}{LAut}
\DeclareMathOperator{\Span}{Span}

\theoremstyle{plain}
\newtheorem{theorem}{Theorem}[section]
\newtheorem{corollary}[theorem]{Corollary}
\newtheorem{proposition}[theorem]{Proposition}
\newtheorem{lemma}[theorem]{Lemma}

\theoremstyle{definition}
\newtheorem{definition}[theorem]{Definition}
\newtheorem{remark}[theorem]{Remark}
\newtheorem{example}[theorem]{Example}

\crefrangeformat{equation}{#3(#1)#4--#5(#2)#6}

\crefname{theorem}{Theorem}{Theorems}
\crefname{lemma}{Lemma}{Lemmas}
\crefname{corollary}{Corollary}{Corollaries}
\crefname{proposition}{Proposition}{Propositions}
\crefname{definition}{Definition}{Definitions}
\crefname{example}{Example}{Examples}
\crefname{remark}{Remark}{Remarks}
\crefname{conjecture}{Conjecture}{Conjectures}
\crefname{question}{Question}{Questions}
\crefname{section}{Section}{Sections}
\crefname{equation}{\unskip}{\unskip}
\crefname{enumi}{\unskip}{\unskip}
\crefname{subsection}{Subsection}{Subsections}

\newcommand{\0}{\theta}

\newcommand{\G}{\Gamma}
\newcommand{\af}{\alpha}
\newcommand{\bt}{\beta}
\newcommand{\lb}{\lambda}

\newcommand{\vf}{\varphi}

\newcommand{\dl}{\delta}

\newcommand{\sg}{\sigma}

\newcommand{\m}{{}^{-1}}
\newcommand{\sst}{\subseteq}
\newcommand{\impl}{\Rightarrow}
\newcommand{\gen}[1]{\langle #1\rangle}
\newcommand{\lf}{\lfloor}
\newcommand{\rf}{\rfloor}
\newcommand{\wtl}{\widetilde}
\newcommand{\ch}{\mathrm{char}}

\newcommand{\id}{\mathrm{id}}

\begin{document}
	\title{Lie automorphisms of incidence algebras}
	
	\author{\'Erica Z. Fornaroli}
	\address{Departamento de Matem\'atica, Universidade Estadual de Maring\'a, Maring\'a, PR, CEP: 87020--900, Brazil}
	\email{ezancanella@uem.br}
	
	\author{Mykola Khrypchenko}
	\address{Departamento de Matem\'atica, Universidade Federal de Santa Catarina,  Campus Reitor Jo\~ao David Ferreira Lima, Florian\'opolis, SC, CEP: 88040--900, Brazil}
	\email{nskhripchenko@gmail.com}
	
	\author{Ednei A. Santulo Jr.}
	\address{Departamento de Matem\'atica, Universidade Estadual de Maring\'a, Maring\'a, PR, CEP: 87020--900, Brazil}
	\email{easjunior@uem.br}
	
	
	\begin{abstract}
		Let $X$ be a finite connected poset and $K$ a field. We give a full description of the Lie automorphisms of the incidence algebra $I(X,K)$. In particular, we show that they are in general not proper.
	\end{abstract}
	
	\maketitle
	
	\section*{Introduction}
	
	Let $A$ be an associative algebra over a commutative ring $R$. Then $A$ becomes a Lie algebra under the \textit{commutator} $[a,b]=ab-ba$ of $a,b\in A$. By a \textit{Lie automorphism} of $A$ one means an automorphism of the Lie algebra $(A,[\phantom{a},\phantom{a}])$. Clearly, an automorphism $\vf$ of $A$ and the negative $-\psi$ of an anti-automorphism $\psi$ of $A$ are Lie automorphisms of $A$. Moreover, if $\nu$ is an $R$-linear central-valued map on $A$ such that $\nu([A,A])=\{0\}$ and $\phi$ is either $\vf$ or $-\psi$ as above, then $\phi+\nu$ is a Lie endomorphism of $A$, which, under certain assumptions on $\nu$, becomes bijective. We will call such Lie automorphisms {\it proper}. Hua~\cite{Hua51} proved that each Lie automorphism of the full matrix ring $M_n(R)$, $n>2$, over a division ring $R$, $\ch(R)\not\in\{2,3\}$, is proper. Martindale extended this result to Lie isomorphisms between primitive rings~\cite{Martindale63}, simple rings~\cite{Martindale69simple} and, yet more generally, between prime rings~\cite{Martindale69}. {\DJ}okovi\'c~\cite{dzD} described the group of Lie automorphisms of the upper triangular matrix algebra $T_n(R)$, where $R$ is a commutative unital ring with trivial idempotents. It follows from his description that each Lie automorphisms of $T_n(R)$ is proper (see \cref{Laut-T_n(K)}). Cao~\cite{Cao97} generalized the result by {\DJ}okovi\'c to the case of commutative rings without any restriction on idempotents. A similar result has also been independently proved by Marcoux and Sourour~\cite{Marcoux-Sourour99}. Cecil~\cite{Cecil} showed that Lie isomorphisms between block-triangular matrix algebras over a UFD are proper.
	
	The incidence algebra $I(X,R)$ of a locally finite poset $X$ over a commutative unital ring $R$ is a natural generalization of $T_n(R)$. Jordan and Lie maps on $I(X,R)$ (and even on more general algebras) have been actively studied during the last 5 years (see~\cite{Akkurts-Barker,BFK,BFK2,KW,WangXiao19,XiaoYang20,Zhang-Khrypchenko}). Usually, all Lie-type maps on $I(X,R)$ are proper. This is no longer the case for Lie automorphisms of $I(X,K)$, where $K$ is a field and $X$ is finite and connected, as we show in this paper. In \cref{sec-prelim} we recall basic definitions about posets and their incidence algebras. In \cref{sec-J_i,sec-ideals} we prove several facts on the structure of $(I(X,K),[\phantom{a},\phantom{a}])$ which will be used throughout the work. In \cref{sec-LAut-decomp} we reduce the description of all Lie automorphisms of $I(X,K)$ to the description of those which we call {\it elementary} (see \cref{LAut-cong-Inn_1-rtimes-wtl-LAut}). \cref{sec-elementary} is the main part of the paper. We describe elementary Lie automorphisms of $I(X,K)$ in terms of triples $(\0,\sg,c)$, where $\0$ is a bijection of $B=\{e_{xy} : x<y\}$, monotone on maximal chains in $X$ and satisfying a certain combinatorial condition which involves cycles in $X$, $\sg$ is a ``$1$-cocycle-looking'' map related to $\0$ and $c$ is a sequence of $|X|$ elements of $K$ such that $\sum c_i\ne 0$ (see \cref{vf-decomp-as-tau_0_sg_c}).
	
	\section{Preliminaries}\label{sec-prelim}

	\subsection{Posets}
	
	Let $(X,\le)$ be a partially ordered set (poset, for short) and let $x,y\in X$. The \emph{interval} from $x$ to $y$ will be denoted by $\lfloor x,y \rfloor$, that is, $\lfloor x,y \rfloor=\{z\in X : x\leq z\leq y\}$. If all the intervals of $X$ are finite, then $X$ is said to be \textit{locally finite}. A \textit{chain} in $X$ is a linearly ordered (under $\le$) subset of $X$. The \textit{length} of a finite chain $C\sst X$ is $|C|-1$. More generally, the \textit{length} of a finite non-empty subset $Y\sst X$, denoted by $l(Y)$, is the maximum length of chains $C\sst Y$. Obviously, $l(Y)\le |Y|-1$. A {\it walk} in $X$ is a sequence $x_0,x_1,\dots,x_m$ of elements of $X$, such that $x_i$ is comparable with $x_{i+1}$ and, moreover, $l(\lf x_i,x_{i+1}\rf)=1$ (if $x_i\le x_{i+1}$) or $l(\lf x_{i+1},x_i\rf)=1$ (if $x_{i+1}\le x_i$) for all $i=0,\dots,m-1$. A walk $x_0,x_1,\dots,x_m$ is said to be {\it closed} if $x_0=x_m$. A {\it path} is a walk in which $x_i\ne x_j$ for $i\ne j$. A {\it cycle} is a closed walk $x_0,x_1,\dots,x_m=x_0$ such that $m\ge 4$ and $x_i=x_j\impl \{i,j\}=\{0,m\}$  for $i\ne j$. The set $X$ is said to be {\it connected} if for any pair of $x,y\in X$ there is a path $x=x_0,\dots,x_m=y$. In this work $X$ will always be a connected finite poset.
	
	\subsection{Incidence algebras}
	
	Let $X$ be a locally finite poset and let $K$ be a field. The \emph{incidence algebra} $I(X,K)$ of $X$ over $K$ is the set 
	$I(X,K)=\{f:X\times X\to K : f(x,y)=0 \text{ if } x\nleq y\}$ endowed with the usual structure of a vector space over $K$ and the product defined by
	$$(fg)(x,y)=\sum_{x\leq t\leq y}f(x,t)g(t,y),$$
	for any $f, g\in I(X,K)$. Then $I(X,K)$ is a $K$-algebra with identity $\delta$ given by
	\begin{align*}
	\delta(x,y)=
	\begin{cases}
	1, & x=y,\\
	0, & x\ne y.
	\end{cases}
	\end{align*} 
	Moreover, if $X$ is finite, which is the case we deal with in this paper, $I(X,K)$ admits the standard basis $\{e_{xy} : x\leq y\}$, where
	\begin{align*}
	e_{xy}(u,v)=
	\begin{cases}
	1, & (u,v)=(x,y),\\
	0, & (u,v)\ne(x,y).
	\end{cases}
	\end{align*} 
	Indeed, $f=\displaystyle\sum_{x\le y}f(x,y)e_{xy}$ for all $f\in I(X,K)$. We will denote $e_{x}=e_{xx}$.
	
	Let $B=\{e_{xy} : x<y\}$. It is known (see~\cite[Theorem 4.2.5]{SpDo}) that the Jacobson radical of $I(X,K)$ is 
	$$
	J(I(X,K))=\{f\in I(X,K) : f(x,x)=0 \text{ for all } x\in X\}=\Span_K B.
	$$
	An element $f\in I(X,K)$ is said to be \emph{diagonal}, if $f(x,y)=0$ for $x\neq y$. Diagonal elements form a commutative subalgebra of $I(X,K)$ spanned by $\{e_{x} : x \in X\}$, which we denote by $D(X,K)$. Clearly, each $f\in I(X,K)$ can be uniquely written as $f=f_D+f_J$ with $f_D\in D(X,K)$ and $f_J\in J(I(X,K))$. 
	
	We denote by $\Aut(I(X,K))$ and $\laut(I(X,K))$ the groups of usual and Lie automorphisms of $I(X,K)$, respectively. We will also deal with the group of bijective $K$-linear maps of $I(X,K)$, denoted by $\mathrm{GL}(I(X,K))$, and with the group 
	$\inn_1(I(X,K))$ of the inner automorphisms of $I(X,K)$ consisting of the conjugations by $\bt\in I(X,K)$ with $\bt_D=\dl$.

	\section{The lower central series of $J(I(X,K))$}\label{sec-J_i}
	
	From now on $K$ will be a field and $X$ a finite connected poset of cardinality $n$.
	
	\begin{lemma}\label{comuta_e}
		An element $f\in I(X,K)$ commutes with $e_{uv}$ if and only if $f(x,u)=0$ for $x<u$, $f(v,y)=0$ for $y>v$ and $f(u,u)=f(v,v)$.
	\end{lemma}
	\begin{proof}
		We have
		\begin{align*}
		fe_{uv}=e_{uv}f & \Leftrightarrow \sum_{x\leq y}f(x,y)e_{xy}e_{uv}=\sum_{x\leq y}f(x,y)e_{uv}e_{xy}\\
		& \Leftrightarrow \sum_{x\le u\le v}f(x,u)e_{xv}=\sum_{u\le v\le y}f(v,y)e_{uy},
		\end{align*}
		whence the desired property of $f$ follows.
	\end{proof}
	
	We recall that the \emph{derived ideal} of a $K$-Lie algebra $L$ is 
	$$[L,L]=\Span_K\{[a,b] : a,b\in L\}.$$
	The sequence of Lie ideals
	$$L\supseteq [L,L]\supseteq [L,[L,L]]\supseteq \cdots $$
	is called \emph{the lower central series} of $L$. It is clear that each term of this series is invariant under any automorphism of $L$.
	
	Let $J_1$ be the derived ideal of $I(X,K)$ and
	$$
	J_m=[J_1,J_{m-1}]=\Span_K\{[f,g] : f\in J_1, g\in J_{m-1}\},\ m\ge 2.
	$$
	We thus have the following.
	\begin{proposition}\label{varphi_Jm}
		If $\varphi\in\laut(I(X,K))$, then $\varphi(J_m)=J_m$ for all $m\geq 1$.
	\end{proposition}
	
	\begin{proposition}\label{J}
		The ideal $J_1$ coincides with $J(I(X,K))$.
	\end{proposition}
	\begin{proof}
		Let $f,g\in I(X,K)$. For all $x\in X$, we have
		$$(fg-gf)(x,x)=f(x,x)g(x,x)-g(x,x)f(x,x)=0.$$
		Thus $[f,g]\in J(I(X,K))$.
		
		On the other hand, if $x<y$ in $X$, then
		$
		e_{xy}=e_{x}e_{xy}-e_{xy}e_{x}=[e_{x},e_{xy}]\in J_1.
		$	
	\end{proof}
	
	\begin{proposition}\label{J_m}
		Let $m$ be a positive integer. Then 
		$$
		J_m=\Span_K\{e_{xy} : l(\lfloor x,y \rfloor) \geq m\}=J(I(X,K))^m. 
		$$
		In particular, $J_{n}=J(I(X,K))^n=\{0\}$.
		It follows that $J_m$ is an ideal (and, therefore, a Lie ideal) of $I(X,K)$.
	\end{proposition}
	\begin{proof}
		The statement is obvious for $m=1$ by \cref{J} so let $m>1$.
		
		Let $e_{u_1v_1},\ldots, e_{u_mv_m}\in J(I(X,K))$, $i=1,\ldots,m$.  If $e_{u_1v_1}\cdots e_{u_mv_m}\neq 0$, then $v_i=u_{i+1}$, $i=1,\ldots, m-1$. In this case, $u_1<u_2<\cdots<u_m<v_m$ and $e_{u_1v_1}\cdots e_{u_mv_m}=e_{u_1v_m}$ with $l(\lfloor u_1,v_m\rfloor)\geq m$. Therefore,
		$$
		J(I(X,K))^m\subseteq \Span_K\{e_{xy} : l(\lfloor x,y \rfloor) \geq m\}.
		$$
		
		Assume that $\Span_K\{e_{xy} : l(\lfloor x,y \rfloor) \geq k\}\subseteq J_k$ for some positive integer $k$ and let $x<y$ such that $l(\lfloor x,y\rfloor)\geq k+1$. We will show that $e_{xy}\in J_{k+1}$. Let $x=t_1<\cdots<t_{k+2}=y$ be a chain in $\lfloor x,y \rfloor$ of length $k+1$. Then
		$$e_{xy}=e_{xt_{k+1}}e_{t_{k+1}y}-e_{t_{k+1}y}e_{xt_{k+1}}=[e_{xt_{k+1}},e_{t_{k+1}y}].$$
		Observe that $e_{t_{k+1}y}\in J(I(X,K))$ and $l(\lfloor x,t_{k+1} \rfloor)\geq k$. By the induction hypothesis, $e_{xt_{k+1}}\in J_k$ and, therefore, $e_{xy}\in J_{k+1}$, as we wanted. Thus,
		$$
		\Span_K\{e_{xy} : l(\lfloor x,y \rfloor) \geq m\}\subseteq J_m.
		$$
		
		Finally, let $J_k\subseteq J(I(X,K))^k$ for some positive integer $k$. If $f\in J(I(X,K))$ and $g\in J_k$, then $[f,g]\in J(I(X,K))^{k+1}$. Thus,	$J_m\subseteq J(I(X,K))^m$.	
	\end{proof}
	
	\cref{J_m} implies that the lower central series of $(J(I(X,K)),[\ ,\ ])$ is
	$$J(I(X,K))\supseteq J(I(X,K))^2\supseteq \cdots \supseteq J(I(X,K))^{n-1}\supseteq J(I(X,K))^n=\{0\}.$$
	
	\begin{proposition}\label{Z}
		The center $Z$ of $J(I(X,K))$ is
		$$Z=\Span_K\{e_{xy} : x<y, x \text{ is minimal and } y \text{ is maximal}\}.$$
	\end{proposition}
	\begin{proof}
		Let $u,v,x,y \in X$ such that $u<v$, $x<y$, $x$ is minimal and $y$ is maximal. Then $[e_{xy},e_{uv}]=\delta_{yu}e_{xv}-\delta_{vx}e_{uy}$.
		If $y=u$, then $y<v$, which contradicts the maximality of $y$. If $v=x$, then $u<x$, which contradicts the minimality of $x$. Thus,
		$[e_{xy},e_{uv}]=0$ and, therefore, $e_{xy}\in Z$.
		
		Conversely, let $g=\sum_{x<y}g(x,y)e_{xy}\in Z$. Assume that $x$ is not minimal. Then there exists $u<x$. Using $ge_{ux}=e_{ux}g$ we conclude that $g(x,y)=0$ for all $y>x$ by \cref{comuta_e}. Similarly, if $y$ is not maximal, $g(x,y)=0$ for all $x<y$.
	\end{proof}

	\section{Associative and Lie ideals of $I(X,K)$}\label{sec-ideals}
	
	We begin this section by presenting some properties of associative and Lie ideals of $I(X,K)$ contained in $J(I(X,K))$.
	
	\begin{lemma}\label{exytala}
		Let $I$ be a Lie ideal of $I(X,K)$ such that $I\subseteq J(I(X,K))$. If $f\in I$ and $f(x,y)\neq 0$, then $e_{xy}\in I$.
	\end{lemma}
	\begin{proof}
		Let $f\in I$ such that $f(x,y)\neq 0$. Then $x<y$ since $I\subseteq J(I(X,K))$. Therefore, $f(x,y)e_{xy}=[[e_x,f],e_y]\in I$, whence $e_{xy}\in I$.
	\end{proof}
	
	\begin{corollary}\label{exytalamemo}
		Let $I$ be a Lie ideal of $I(X,K)$ such that $I\subseteq J(I(X,K))$. Then
		$I=\Span_K\{e_{xy} : e_{xy} \in I\}.$\footnote{For an associative ideal this is well-known (see~\cite[Lemma 3.1]{Doubilet-Rota-Stanley72}).}
	\end{corollary}
	
	Given a subset $S$ of $I(X,K)$, we will denote by $\langle S\rangle$ (resp. $\langle S\rangle_L$) the ideal (resp. Lie ideal) of $I(X,K)$ generated by $S$.
	
	\begin{lemma}\label{<e_xy>-description}
		For all $x\le y$ one has $\langle e_{xy}\rangle=\Span_K\{e_{uv}: u\le x\le y\le v\}$.
	\end{lemma}
	\begin{proof}
		If $u\le x\le y\le v$, then $e_{uv}=e_{ux}e_{xy}e_{yv}$. Conversely, the product $e_{ab}e_{xy}e_{cd}$ is nonzero if and only if $b=x$ and $y=c$, in which case it equals $e_{ad}$ with $a\le x\le y\le d$.
	\end{proof}
	
	\begin{lemma}\label{f(uv)-ne-0-for-x<=u<v<=y}
		Let $e_{xy}\in \langle S\rangle$, where $S\subseteq J(I(X,K))$. Then there are $f\in S$ and $x\leq u<v\leq y$ such that $f(u,v)\neq 0$.
	\end{lemma}
	\begin{proof}
		Write $e_{xy}=\sum_i g_if_ih_i$ with $g_i,h_i\in I(X,K)$ and $f_i\in S$. If $f_i(u,v)=0$ for all $x\leq u<v\leq y$, then $1=e_{xy}(x,y)=\sum_i\sum_{x\leq u<v\leq y}g_i(x,u)f_i(u,v)h_i(v,y)=0$, a contradiction.
	\end{proof}
	
	\begin{corollary}\label{ideal}
		If $S\subseteq J(I(X,K))$, then $\langle S\rangle=\langle S\rangle_L$. Therefore, every Lie ideal of $I(X,K)$ contained in $J(I(X,K))$ is an ideal.
	\end{corollary}
	\begin{proof}
		Clearly, $\langle S\rangle_L\subseteq\langle S\rangle$. To prove that $\langle S\rangle\subseteq \langle S\rangle_L$, we only need to show by \cref{exytalamemo} that  $e_{xy}\in \langle S\rangle_L$ whenever $e_{xy}\in \langle S\rangle$. Let $e_{xy}\in \langle S\rangle$. By \cref{f(uv)-ne-0-for-x<=u<v<=y} there are $f\in S$ and $x\leq u<v\leq y$ such that $f(u,v)\neq 0$. Then $e_{xy}=f(u,v)^{-1}[[e_{xu},f],e_{vy}]\in \langle S\rangle_L$.
	\end{proof}

	\begin{proposition}\label{dim}
		Let $I\subseteq J_m$ be an ideal of $I(X,K)$ such that $\dim\dfrac{I}{I\cap J_{m+1}}=1$. Then there exists a unique $e_{xy}\in J_m-J_{m+1}$ such that any $f\in I-J_{m+1}$ is of the form $f=f(x,y)e_{xy}+j_{xy}$, where $j_{xy}\in J_{m+1}$.
	\end{proposition}
	\begin{proof}
		Let $f\in I-J_{m+1}$. Then there is $e_{xy}\in J_m-J_{m+1}$ such that $f(x,y)\neq 0$. By \cref{exytala}, $e_{xy}\in I$. Since $\dim\dfrac{I}{I\cap J_{m+1}}=1$, $e_{xy}$ is the only $e_{uv}\in J_m-J_{m+1}$ that belongs to $I$. Thus, $f-f(x,y)e_{xy}\in J_{m+1}$. 
	\end{proof}

	\begin{lemma}\label{<e_xy><e_uv>-is-zero}
		Let $x\le y$ and $u\le v$. If $y\not\le u$, then $\langle e_{xy}\rangle\langle e_{uv}\rangle=\{0\}$. 
	\end{lemma}
	\begin{proof}
		Let $f\in\langle e_{xy}\rangle$ and $g\in\langle e_{uv}\rangle$.
		If there are $a\le b$ such that
		\begin{align*}
		0\ne (fg)(a,b)=\sum_{a\le c\le b}f(a,c)g(c,b),
		\end{align*}
		then $f(a,c)\ne 0$ and $g(c,b)\ne 0$ for some $a\le c\le b$. It follows by \cref{<e_xy>-description} that $a\le x\le y\le c$ and $c\le u\le v\le b$, whence $y\le u$, a contradiction.
	\end{proof}

	\section{The decomposition of $\laut(I(X,K))$ into a semidirect product}\label{sec-LAut-decomp}
	
	Let $L_i=\Span_K\{e_{xy}:e_{xy}\in J_i - J_{i+1}\}= \Span_K \{e_{xy} : l(\lfloor x,y\rfloor) = i\}$, for any integer $i\ge 0$, where $J_0:=I(X,K)$. Then $J_i=\bigoplus_{k\ge i} L_k$ as $K$-vector spaces, $i\ge 0$. 
	\begin{definition}
		Given $\vf\in\laut(I(X,K))$, set $\wtl\vf:I(X,K)\to I(X,K)$ to be the $K$-linear map defined as follows: for any $e_{xy}$, if $e_{xy}\in L_i$, then
		\begin{align*}
		\vf(e_{xy})=\wtl\vf(e_{xy})+j_{xy},
		\end{align*}
		where $\wtl\vf(e_{xy})\in L_i$ and $j_{xy}\in J_{i+1}$.
	\end{definition} 
	Clearly, $\wtl\vf$ is well-defined in view of \cref{varphi_Jm}.
	
	\begin{proposition}\label{wtl-homo}
		The map $\vf\mapsto\wtl\vf$ is a group homomorphism from $\laut(I(X,K))$ to $\mathrm{GL}(I(X,K))$.
	\end{proposition}
	\begin{proof}
		We first prove that $\wtl\vf\in\mathrm{GL}(I(X,K))$. Since $\widetilde{\varphi}$ is linear, it suffices to check that $\widetilde{\varphi}$ is onto. Let $x\le y$ and take $g\in I(X,K)$ such that $\varphi(g)=e_{xy}$. If $e_{xy}\in L_i$, then, by \cref{varphi_Jm}, $g\in J_i$, so $g=f+h$ with $f\in L_i$ and $h\in J_{i+1}$. It follows that
		$
		e_{xy}=\vf(g)=\vf(f)+\vf(h)=\wtl\vf(f)+j+\vf(h),
		$
		where $j\in J_{i+1}$. Thus, $e_{xy}-\wtl\vf(f)=j+\vf(h)$ in which $e_{xy}-\wtl\vf(f)\in L_i$ and $j+\vf(h)\in J_{i+1}$. Therefore, $\wtl\vf(f)=e_{xy}$ (and $j+\vf(h)=0$), which implies that $\wtl\vf$ is onto.
		
		Let $\vf,\psi\in\laut(I(X,K))$. We need to prove that
		$
		\wtl{\vf\circ\psi}=\wtl\vf\circ\wtl\psi.
		$
		Let $x\le y$ with $e_{xy}\in L_i$. Then $\psi(e_{xy})=\wtl\psi(e_{xy})+j_{xy}$, where $j_{xy}\in J_{i+1}$, so $\vf(\psi(e_{xy}))=\vf(\wtl\psi(e_{xy}))+\vf(j_{xy})=\wtl\vf(\wtl\psi(e_{xy}))+j'_{xy}+\vf(j_{xy})$ for some $j'_{xy}\in J_{i+1}$. Since $\wtl\vf(\wtl\psi(e_{xy}))\in L_i$ and $j'_{xy}+\vf(j_{xy})\in J_{i+1}$, then $(\wtl{\vf\circ\psi})(e_{xy})=\wtl\vf(\wtl\psi(e_{xy}))$.
	\end{proof}
	
	Let us introduce $B_i=B\cap J_i$ and $X^2_<=\{(x,y)\in X^2: x<y\}$.
	
	\begin{lemma}\label{exypermutation-for-wtl-vf}
		Let $\varphi\in\laut(I(X,K))$. There exist a bijection $\theta=\0_\vf:B\to B$ with $\0(B_i)\sst B_i$ and a map $\sg=\sg_\vf:X^2_<\to K^*$ such that, for any $e_{xy}\in B$,
		\begin{align}\label{wtl-vf(e_xy)=sg.e_uv}
		\wtl\varphi(e_{xy})=\sigma(x,y)\0(e_{xy}).
		\end{align}
	\end{lemma}
	\begin{proof}
		Suppose that $e_{xy}\in B_i-B_{i+1}$. We will introduce
		$
		I_{xy}=\langle e_{xy}\rangle/(\langle e_{xy}\rangle\cap J_{i+1}).
		$ 
		Observe that $\dim(I_{xy})=1$ since any $e_{uv}\in\gen{e_{xy}}$ with $e_{uv}\ne e_{xy}$ belongs to $J_{i+1}$ by \cref{<e_xy>-description}.
		
		Note that $\varphi(\langle e_{xy}\rangle)\subseteq J_i$ by \cref{varphi_Jm} and $\varphi(\langle e_{xy}\rangle)$ is an ideal of $I(X,K)$ by \cref{ideal}. Moreover, $\varphi(\langle e_{xy}\rangle)/(\varphi(\langle e_{xy}\rangle)\cap J_{i+1})$ and $I_{xy}$ are isomorphic as Lie algebras. Hence
		$
		\dim (\varphi(\langle e_{xy}\rangle)/(\varphi(\langle e_{xy}\rangle)\cap J_{i+1})) = \dim(I_{xy})=1.
		$
		By \cref{dim} there are $e_{uv}\in B_i-B_{i+1}$, $k\in K^*$ and $g\in J_{i+1}$ such that $\varphi(e_{xy})=k e_{uv}+g$. Hence, $\wtl\varphi(e_{xy})=k e_{uv}$. We thus define $\theta:B\to B$ by $\0(e_{xy})=e_{uv}$ and $\sg:X^2_<\to K^*$ by $\sg(x,y)=k$. Clearly, $\0(B_i)\sst B_i$.
		
		It remains to prove that $\theta$ is a bijection. Since $B$ is finite, it suffices to establish the injectivity of $\0$. Suppose that  $\wtl\varphi(e_{xy}) = \sigma(x,y) e_{uv}$  and $\wtl\varphi(e_{ab}) = \sg(a,b) e_{uv}$ for some $e_{xy},e_{ab}\in B_i-B_{i+1}$. Then $\varphi(e_{xy}) = \sigma(x,y) e_{uv}+g$  and $\varphi(e_{ab}) = \sg(a,b) e_{uv}+h$ for some $g,h\in J_{i+1}$. Consequently, $\varphi(\sg(a,b) e_{xy} - \sigma(x,y) e_{ab}) = \sg(a,b) g - \sigma(x,y) h\in J_{i+1}$. It follows that $\sg(a,b) e_{xy} - \sigma(x,y) e_{ab}\in J_{i+1}$ by \cref{varphi_Jm}. Since $\sg(x,y),\sigma(a,b)\in K^*$, then $e_{xy}=e_{ab}$ (and $\sg(x,y)=\sigma(a,b)$).
	\end{proof}
	
	\begin{lemma}\label{vf(e_xy)-wtlvf(e_xy)}
		Let $\varphi\in\laut(I(X,K))$ and $e_{xy}\in L_i$, where $i>0$. Then $\vf(e_{xy})-\wtl\vf(e_{xy})\in\gen{\0_\vf(e_{xy})}\cap J_{i+1}$.
	\end{lemma}
	\begin{proof}
		Write $\vf(e_{xy})=\wtl\vf(e_{xy})+g$ with $g\in J_{i+1}$. Since $e_{xy}=[e_x,e_{xy}]$, then
		\begin{align*}
		\wtl\vf(e_{xy})+g =\vf([e_x,e_{xy}])=[\vf(e_x),\vf(e_{xy})]
		=[\vf(e_x),\wtl\vf(e_{xy})]+[\vf(e_x),g].
		\end{align*}
		But $\wtl\vf(e_{xy})\in \gen{\0_\vf(e_{xy})}$ by \cref{exypermutation-for-wtl-vf}, so
		\begin{align}\label{g-[vf(e_x)_g]-in-<e_uv>}
		g-[\vf(e_x),g]=[\vf(e_x),\wtl\vf(e_{xy})]-\wtl\vf(e_{xy})\in\gen{\0_\vf(e_{xy})}.
		\end{align}
		Similarly $e_{xy}=[e_{xy},e_y]$ implies that
		\begin{align}\label{g-[g_vf(e_y)]-in-<e_uv>}
		g-[g,\vf(e_y)]\in\gen{\0_\vf(e_{xy})}.
		\end{align}
		Now, using \cref{g-[vf(e_x)_g]-in-<e_uv>,g-[g_vf(e_y)]-in-<e_uv>} we have
		\begin{align}\label{g-[[vf(e_x)_g]_vf(e_y)]}
		g-[[\vf(e_x),g],\vf(e_y)]=g-[g,\vf(e_y)]+[g-[\vf(e_x),g],\vf(e_y)]\in\gen{\0_\vf(e_{xy})}.
		\end{align}
		Choose $h\in J_{i+1}$, such that $g=\vf(h)$. Then
		\begin{align*}
		[[\vf(e_x),g],\vf(e_y)]=[[\vf(e_x),\vf(h)],\vf(e_y)]=\vf([[e_x,h],e_y])=\vf(h(x,y)e_{xy}).
		\end{align*}
		But $h(x,y)=0$ by \cref{J_m}, so $[[\vf(e_x),g],\vf(e_y)]=0$ and the result follows from \cref{g-[[vf(e_x)_g]_vf(e_y)]}.
	\end{proof}
	
	\begin{proposition}\label{elementary_phi_tilde}
		Let $\vf\in\laut(I(X,K))$. Then $\wtl\vf\in\laut(I(X,K))$.
	\end{proposition}
	
	\begin{proof}
		Since $\widetilde{\varphi}\in\mathrm{GL}(I(X,K))$ by \cref{wtl-homo}, it remains to prove $\widetilde{\varphi}([e_{xy},e_{uv}])=[\widetilde{\varphi}(e_{xy}),\widetilde{\varphi}(e_{uv})]$, for all $x\le y$ and $u\le v$. 
		
		The case $e_{xy},e_{uv}\in L_0=D(X,K)$ is trivial because $[L_0,L_0]=\{0\}$. 
		If $e_{xy}\in L_0$ and $e_{uv}\in L_i$ with $i>0$, then $\vf([e_{xy},e_{uv}])=[\varphi(e_{xy}),\varphi(e_{uv})]=[\widetilde{\varphi}(e_{xy}),\widetilde{\varphi}(e_{uv})]+j$, where $j\in J_{i+1}$. Since $[L_0,L_i]\sst L_i$, then $\wtl\vf([e_{xy},e_{uv}])=[\widetilde{\varphi}(e_{xy}),\widetilde{\varphi}(e_{uv})]$.
		
		Let $e_{xy}\in L_i$ and $e_{uv}\in L_j$ with $i,j>0$. By \cref{exypermutation-for-wtl-vf,vf(e_xy)-wtlvf(e_xy)}, $\vf(e_{xy})=\af e_{ab}+g$ and $\vf(e_{uv})=\bt e_{cd}+h$, where $\af,\bt\in K^*$, $e_{ab} \in L_i$, $e_{cd} \in L_j$, $g\in \gen{e_{ab}}\cap J_{i+1}$, $h\in \gen{e_{cd}}\cap J_{j+1}$. Thus, $\wtl\vf(e_{xy})=\af e_{ab}$ and $\wtl\vf(e_{uv})=\bt e_{cd}$. If $[e_{xy},e_{uv}]=0$, then $[\wtl\vf(e_{xy}),\wtl\vf(e_{uv})]=0$. Indeed, assume that $e_{ab}e_{cd}\ne 0$, i.e., $b=c$. Then $0=[\vf(e_{xy}),\vf(e_{uv})]=\af\bt e_{ad}+p$, where $p=\af[e_{ab},h]+\bt[g,e_{cd}]+[g,h]$. Observe from \cref{<e_xy>-description} that $p\in \gen{e_{ad}}\cap J_{k+1}$, where $k=l(\lf a,d\rf)$, a contradiction. The case $e_{cd}e_{ab}\ne 0$ is similar. Let now $[e_{xy},e_{uv}]\ne 0$. Then applying the above argument to $e_{ab},e_{cd}$ and $\vf\m$ we conclude that $[e_{ab},e_{cd}]\ne 0$, because $\af\m e_{xy}=\wtl{\vf\m}(e_{ab})$ and $\bt\m e_{uv}=\wtl{\vf\m}(e_{cd})$. Let $b=c$ and write $\vf([e_{xy},e_{uv}])=[\vf(e_{xy}),\vf(e_{uv})]=\af\bt e_{ad}+p$ as above. If $l(\lf a,d\rf)=k$, then $p\in J_{k+1}$, whence $\wtl\vf([e_{xy},e_{uv}])=\af\bt e_{ad}=[\wtl\vf(e_{xy}),\wtl\vf(e_{uv})]$. If $a=d$, the proof is analogous. 
	\end{proof}
	
	\begin{corollary}\label{wtl-endo}
		The map $\vf\mapsto\wtl\vf$ is an endomorphism of the group $\laut(I(X,K))$.
	\end{corollary}
	
	\begin{definition}
		Denote by $\wtl\laut(I(X,K))$ the image of $\vf\mapsto\wtl\vf$. A Lie automorphism from $\wtl\laut(I(X,K))$ will be called \textit{elementary}. 
	\end{definition}
	
	\begin{remark}\label{rem_elementary}
		Observe that $\vf$ is elementary if and only if $\vf(L_i)\sst L_i$, for all $i\geq 0$.
	\end{remark}
	
	We will also introduce the temporary notation $N$ for the kernel of $\vf\mapsto\wtl\vf$. We will soon see that $N=\inn_1(I(X,K))$.
	
	\begin{remark}\label{wtl-of-wtl-is-wtl}
		For any $\vf\in\laut(I(X,K))$ one has $\wtl{\wtl\vf}=\wtl\vf$. In particular, $\vf\circ(\wtl\vf)^{-1}$ belongs to $N$.
	\end{remark}

	\begin{lemma}\label{vf(e_x)-in-<e_x>}
		Let $\vf\in N$. Then $\vf(e_x)\in\langle e_x\rangle$ for all $x\in X$.
	\end{lemma}
	\begin{proof}
		Take $u<v$ such that $x\not\in \lfloor u,v\rfloor$. Write $\vf(e_x)=e_x+j_x$ and $\vf(e_u)=e_u+j_u$, where $j_x,j_u\in J(I(X,K))$. We need to show that $j_x(u,v)=0$. Applying $\vf$ to $[e_x,e_u]=0$, we obtain
		\begin{align*}
		[e_x,j_u]+[j_x,e_u]+[j_x,j_u]=0.
		\end{align*}
		Calculating the value of this sum at $(u,v)$, we have
		\begin{align}\label{j_x(uv)=[j_u_j_x](uv)}
		j_x(u,v)=[j_x,j_u](u,v)=\sum_{u<w<v}(j_x(u,w)j_u(w,v)-j_u(u,w)j_x(w,v)).
		\end{align}
		We will prove that $j_x(u,v)=0$ by induction on $l(\lfloor u,v\rfloor)$. If $l(\lfloor u,v\rfloor)=1$, then $j_x(u,v)=0$, since there is no $w$ with $u<w<v$ on the right-hand side of \cref{j_x(uv)=[j_u_j_x](uv)}. If $l(\lfloor u,v\rfloor)>1$, then $l(\lfloor u,w\rfloor)$ and $l(\lfloor w,v\rfloor)$ are strictly less than $l(\lfloor u,v\rfloor)$ for any $u<w<v$. Since moreover $x\not\in\lfloor u,w\rfloor$ and $x\not\in\lfloor w,v\rfloor$, by induction hypothesis $j_x(u,w)=j_x(w,v)=0$, whence $j_x(u,v)=0$ by \cref{j_x(uv)=[j_u_j_x](uv)}.
	\end{proof}
	
	\begin{corollary}\label{vf(e_x)vf(e_y)-is-zero}
		Let $\vf\in N$. Then $\vf(e_x)\vf(e_y)=0$ for all $x\ne y$. 
	\end{corollary}
	\begin{proof}
		Notice that $\vf(e_x)$ commutes with $\vf(e_y)$, since $[\vf(e_x),\vf(e_y)]=\vf([e_x,e_y])=0$. If $x\not\le y$, then the result follows from \cref{<e_xy><e_uv>-is-zero,vf(e_x)-in-<e_x>}. If $x<y$, then $y\not\le x$, so $\vf(e_y)\vf(e_x)=0$ by the previous case, whence $\vf(e_x)\vf(e_y)=0$. 
	\end{proof}
	
	\begin{lemma}\label{vf(dl)-is-dl}
		Let $\vf\in N$. Then $\vf(\dl)=\dl$.
	\end{lemma}
	\begin{proof}
		Since $\dl$ is central, then $\vf(\dl)$ is also central, in particular, it is diagonal, by \cite[Corollary~1.3.15]{SpDo}. On the other hand, $\dl=\sum_{x\in X}e_x$, whence $\vf(\dl)=\vf(\dl)_D=\sum_{x\in X}\vf(e_x)_D=\sum_{x\in X}e_x=\dl$.
	\end{proof}
	
	\begin{corollary}\label{vf(e_x)-idemp}
		Let $\vf\in N$. Then $\vf(e_x)$ is an idempotent for all $x\in X$.
	\end{corollary}
	\begin{proof}
		By \cref{vf(dl)-is-dl} we have $\dl=\vf(\dl)=\sum_{x\in X}\vf(e_x)$. Multiplying this by $\vf(e_x)$ and using \cref{vf(e_x)vf(e_y)-is-zero} we obtain $\vf(e_x)^2=\vf(e_x)$.
	\end{proof}
	
	\begin{proposition}\label{N-is-Inn_1}
		The group $N$ coincides with $\inn_1(I(X,K))$.
	\end{proposition}
	\begin{proof}
		If $\bt=\dl+\rho_1$ with $\rho_1\in J_1$ then $\bt^{-1}=\dl+\rho_2$ with $\rho_2\in J_1$. For any $e_{xy}\in L_i$ we have $\bt e_{xy}\bt^{-1}=(\dl+\rho_1)e_{xy}(\dl+\rho_2)=e_{xy}+\rho$, where $\rho\in J_{i+1}$. Hence, $\inn_1(I(X,K))\sst N$. We will prove the converse inclusion.
		
		Let $\vf\in N$. Define $\bt=\sum_{x\in X}\vf(e_x)e_x$. Since  $\wtl\vf=\id$, we have $\vf(e_x)_D=e_x$, whence $\bt_D=\sum_{x\in X}e_x=\dl$. Clearly, $\bt e_x=\vf(e_x)e_x$. In view of \cref{vf(e_x)-idemp,vf(e_x)vf(e_y)-is-zero} we similarly obtain $\vf(e_x)\bt=\vf(e_x)e_x$. Thus, $\vf(e_x)=\bt e_x\bt^{-1}$. 
		
		Denote by $\xi_\bt$ the conjugation by $\bt$. Then $\xi_\bt\in\inn_1(I(X,K))\sst N$. Set $\psi:=(\xi_\bt)^{-1}\circ\vf$. Then $\psi\in N$ and moreover, $\psi(e_x)=e_x$ for all $x\in X$. Applying $\psi$ to $e_{xy}=[e_x,[e_{xy},e_y]]\in L_i$ with $i>0$, we have
		\begin{align*}
		\psi(e_{xy})=[e_x,[\psi(e_{xy}),e_y]]=[e_x,[e_{xy}+j_{xy},e_y]]=(1+j_{xy}(x,y))e_{xy}=e_{xy},
		\end{align*}
		as $j_{xy}\in J_{i+1}$. Thus, $\psi=\id$ and so $\vf=\xi_\bt$.
	\end{proof}
	
	As an immediate consequence of \cref{wtl-of-wtl-is-wtl,N-is-Inn_1} we have the following.
	\begin{theorem}\label{LAut-cong-Inn_1-rtimes-wtl-LAut}
		The group $\laut(I(X,K))$ is isomorphic to the semidirect product $\inn_1(I(X,K))\rtimes\wtl\laut(I(X,K))$.
	\end{theorem}

	\section{Elementary Lie automorphisms and admissible pairs}\label{sec-elementary}
	
	Let $\varphi\in\wtl\laut(I(X,K))$ and consider $\theta=\0_\vf$ and $\sg=\sg_\vf$ as in \cref{exypermutation-for-wtl-vf}. Then, for any $e_{xy}\in B$,	
	\begin{align}\label{exypermutation}
	\varphi(e_{xy})=\sigma(x,y)\0(e_{xy}).
	\end{align}
	
	
	\begin{remark}\label{0_vf-inv}
		Let $\varphi\in\wtl\laut(I(X,K))$. Then $\0_{\vf\m}=(\0_\vf)\m$ and $\sg_{\vf\m}(u,v)=\sg_\vf(x,y)\m$, where $\0_\vf(e_{xy}) = e_{uv}$.
	\end{remark}
	
	\begin{lemma}\label{0_vf(e_xy)-and-0_vf(e_yz)}
		Let $\varphi\in\wtl\laut(I(X,K))$, $\0=\0_\vf$, $\sg=\sg_\vf$ and $x<y<z$. Then
		\begin{align*}
		\text{either }\0(e_{xz})&=\0(e_{xy})\0(e_{yz})\text{ in which case }\sg(x,z)=\sg(x,y)\sg(y,z),\\
		\text{or }\0(e_{xz})&=\0(e_{yz})\0(e_{xy})\text{ in which case }\sg(x,z)=-\sg(x,y)\sg(y,z).
		\end{align*}
	\end{lemma}
	\begin{proof}
		By \cref{exypermutation} we have
		\begin{align}\label{sg(xz)e_xz=sg(xy)sg(yz)[0(e_xy)0(e_yz)]}
		\sg(x,z)\0(e_{xz})=\vf(e_{xz})=[\vf(e_{xy}),\vf(e_{yz})]=\sg(x,y)\sg(y,z)[\0(e_{xy}),\0(e_{yz})].
		\end{align}
		Since $[\0(e_{xy}),\0(e_{yz})]\ne 0$ and $\0(e_{xy}),\0(e_{yz})\in B$, then either $\0(e_{xy})\0(e_{yz})\ne 0$, in which case $[\0(e_{xy}),\0(e_{yz})]=\0(e_{xy})\0(e_{yz})$, or $\0(e_{yz})\0(e_{xy})\ne 0$, in which case $[\0(e_{xy}),\0(e_{yz})]=-\0(e_{yz})\0(e_{xy})$. Thus, the result follows from \cref{sg(xz)e_xz=sg(xy)sg(yz)[0(e_xy)0(e_yz)]}.
	\end{proof}
	
	\begin{definition}
		Let $\0:B\to B$ be a bijection and $\sg:X^2_<\to K^*$ a map. We say that $\sg$ is \textit{compatible} with $\0$ if $\sg(x,z)=\sg(x,y)\sg(y,z)$ whenever $\0(e_{xz})=\0(e_{xy})\0(e_{yz})$, and $\sg(x,z)=-\sg(x,y)\sg(y,z)$ whenever $\0(e_{xz})=\0(e_{yz})\0(e_{xy})$.
	\end{definition}
	The above definition makes sense, as $e_{xy},e_{uv}\in B$ commute if and only if $e_{xy}e_{uv}=e_{uv}e_{xy}=0$.
	
	\begin{lemma}\label{0_vf-on-a-chain}
		Let $\vf\in\wtl\laut(I(X,K))$ and $\0=\0_\vf$. For any maximal chain $C$: $u_1<u_2<\dots<u_m$ in $X$ there exists a maximal chain $D$: $v_1<v_2<\dots<v_m$ in $X$ such that one of the following two conditions holds:
		\begin{enumerate}
			\item $\0(e_{u_iu_j})=e_{v_iv_j}$ for all $1\le i<j\le m$;\label{0_vf(e_u_iu_j)=e_viv_j}
			\item $\0(e_{u_iu_j})=e_{v_{m-j+1}v_{m-i+1}}$ for all $1\le i<j\le m$.\label{0_vf(e_u_iu_j)=e_m-j+1v_m-i+1}
		\end{enumerate}
	\end{lemma}
	\begin{proof}
		We first construct the chain $v_1<v_2<\dots<v_m$ and then prove that it is maximal. If $m=2$, then $\0(e_{u_1u_2})=e_{v_1v_2}$ by \cref{exypermutation}. Let $m>2$ and consider $u_1<u_2<u_3$. By \cref{0_vf(e_xy)-and-0_vf(e_yz)} we have two cases.
		
		\textbf{Case 1.} There are $v_1<v_2<v_3$ such that $\0(e_{u_1u_2})=e_{v_1v_2}$, $\0(e_{u_2u_3})=e_{v_2v_3}$ and $\0(e_{u_1u_3})=e_{v_1v_3}$. If $m=3$, then we are done. If $m>3$, then there are $u<v$, such that $\0(e_{u_3u_4})=e_{uv}$. Applying \cref{0_vf(e_xy)-and-0_vf(e_yz)} to $u_2<u_3<u_4$, we conclude that either $u=v_3$ or $v=v_2$. Now, applying the same lemma to $u_1<u_3<u_4$, we have either $u=v_3$ or $v=v_1$. Since $v_1\ne v_2$, then $u=v_3$. Hence, we may define $v_4:=v$, and we have $v_1<v_2<v_3<v_4$ and $\0(e_{u_3u_4})=e_{v_3v_4}$, $\0(e_{u_2u_4})=e_{v_2v_4}$, $\0(e_{u_1u_4})=e_{v_1v_4}$ by \cref{0_vf(e_xy)-and-0_vf(e_yz)}. Repeating the same argument, we consecutively construct $v_1<v_2<\dots<v_m$ such that $\0(e_{u_iu_j})=e_{v_iv_j}$ for all $1\le i<j\le m$. 
		
		\textbf{Case 2.} 
		There are $v_{m-2}<v_{m-1}<v_m$ such that $\0(e_{u_1u_2})=e_{v_{m-1}v_m}$, $\0(e_{u_2u_3})=e_{v_{m-2}v_{m-1}}$ and $\0(e_{u_1u_3})=e_{v_{m-2}v_m}$. This case is analogous to Case~1. Repeating a similar argument, we consecutively construct $v_1<v_2<\dots<v_m$ such that $\0(e_{u_iu_j})=e_{v_{m-j+1}v_{m-i+1}}$ for all $1\le i<j\le m$.
		
		Now we prove that the constructed chain $v_1<v_2<\dots<v_m$ is maximal. Since $u_1<u_2<\dots<u_m$ is maximal, then $l(\lf u_i,u_{i+1}\rf)=1$ for all $i=1,\dots,m-1$. The latter is equivalent to $e_{u_iu_{i+1}}\in B_1-B_2$. If $\0(e_{u_iu_{i+1}})=e_{v_iv_{i+1}}$, then $e_{v_iv_{i+1}}\in B_1-B_2$ and hence $l(\lf v_i,v_{i+1}\rf)=1$. If $\0(e_{u_iu_{i+1}})=e_{v_{m-i}v_{m-i+1}}$, then we similarly obtain $l(\lf v_{m-i},v_{m-i+1}\rf)=1$ for all $i=1,\dots,m-1$. Both cases imply that there cannot exist $v$ such that $v_i<v<v_{i+1}$ for some $i=1,\dots,m-1$. There cannot also exist $v_0$ such that $v_0<v_1$ and $v_{m+1}$ such that $v_{m+1}>v_m$. Indeed, $u_1$ is minimal and $u_m$ is maximal in $X$, so that $e_{u_1u_m}\in Z$ by \cref{Z}. Hence $e_{v_1v_m}=\sg(u_1,u_m)\m\vf(e_{u_1u_m})\in Z$. Thus, $v_1$ is minimal and $v_m$ is maximal.
	\end{proof}
	
	\begin{definition}
		Let $\0:B\to B$ be a bijection and $C$ a maximal chain in $X$. We say that $\0$ \textit{is increasing (resp. decreasing) on $C$} if \cref{0_vf-on-a-chain}\cref{0_vf(e_u_iu_j)=e_viv_j} (resp. \cref{0_vf-on-a-chain}\cref{0_vf(e_u_iu_j)=e_m-j+1v_m-i+1}) holds. Moreover, we say that $\0$ \textit{is monotone on maximal chains in $X$} if, for any maximal chain $C$, $\0$ is increasing or decreasing on $C$.
	\end{definition}
	
	\begin{remark}\label{0-inv-monotone}
		If a bijection $\0:B\to B$ is monotone on maximal chains in $X$, then so is $\0^{-1}$, since $X$ is finite.
	\end{remark}
	
	\begin{lemma}\label{vf(e_z)-on-diagonal}
		Let $\vf\in\wtl\laut(I(X,K))$, $\0=\0_\vf$, $x<y$ and $\0(e_{xy})=e_{uv}$. Then for all $z\in X$:
		\begin{align}\label{vf(e_z)_uu-vf(e_z)_vv=pm-one-or-zero}
		\vf(e_z)(v,v)-\vf(e_z)(u,u)&=
		\begin{cases}
		-1, & z=x,\\
		1,& z=y,\\
		0, & z\not\in\{x,y\}.
		\end{cases}
		\end{align}
	\end{lemma}
	\begin{proof}
		Let $e_{xy}\in B_m-B_{m+1}$. By \cref{exypermutation}, we have $\vf(e_{xy})=\sg(x,y) e_{uv}$, where $\sg(x,y)\in K^*$ and $e_{uv}\in B_m-B_{m+1}$. Then applying $\vf$ to $[e_x,e_{xy}]=e_{xy}$ and using the fact that $\vf(e_x)\in D(X,K)$, we obtain 
		\begin{align}
		\sg(x,y) e_{uv} = \sg(x,y)[\vf(e_x),e_{uv}]= \sg(x,y)(\vf(e_x)(u,u)-\vf(e_x)(v,v))e_{uv}.\label{sg.e_uv=sg[vf(e_x)_D_e_uv]+smth}
		\end{align}
		Since $\sg(x,y)\ne 0$, we deduce from \cref{sg.e_uv=sg[vf(e_x)_D_e_uv]+smth} that $\vf(e_x)(u,u)-\vf(e_x)(v,v)=1$.
		
		Similarly one gets $\vf(e_y)(u,u)-\vf(e_y)(v,v)=-1$ from $[e_y,e_{xy}]=-e_{xy}$ and $\vf(e_z)(u,u)-\vf(e_z)(v,v)=0$ ($z\not\in\{x,y\}$) from $[e_z,e_{xy}]=0$.
	\end{proof}
	
	\begin{lemma}\label{0_vf=0_psi-and-vf(e_z)_uu=psi(e_z)_uu}
		Fix an arbitrary $u_0\in X$. Given $\vf,\psi\in\wtl\laut(I(X,K))$,
		if $\0_\vf=\0_\psi$ and $\vf(e_z)(u_0,u_0)=\psi(e_z)(u_0,u_0)$ for all $z\in X$, then $\vf(e_z)=\psi(e_z)$ for all $z\in X$.
	\end{lemma}
	\begin{proof}
		Given $v\in X$, choose a walk $\G:u_0,u_1,\dots,u_m=v$ from $u_0$ to $v$. For each $0\le i\le m-1$ let $v_i<w_i$ be such that
		
		\begin{align}\label{vf(e_v_iw_i)=e_u_iu_i+1}
		\0_{\vf}(e_{v_iw_i})=
		\begin{cases}
		e_{u_iu_{i+1}}, & u_i<u_{i+1},\\
		e_{u_{i+1}u_i}, & u_i>u_{i+1}.
		\end{cases}
		\end{align}
		Then
		\begin{align}\label{vf(e_u)_vv=vf(e_u)_u_0u_0+sum}
		\vf(e_z)(v,v)=\vf(e_z)(u_0,u_0)+\sum_{i=0}^{m-1}\Delta_{\G,i}(z),
		\end{align}
		where $\Delta_{\G,i}(z)=\vf(e_z)(u_{i+1},u_{i+1})-\vf(e_z)(u_i,u_i)$. By \cref{vf(e_z)-on-diagonal},
		\begin{align}\label{vf(e_z)_u_(i+1)u_(i+1)-vf(e_z)_u_iu_i}
		\Delta_{\G,i}(z) & =
		\begin{cases}
		-1, & (u_i<u_{i+1}\wedge z=v_i)\vee(u_i>u_{i+1}\wedge z=w_i),\\
		1,& (u_i<u_{i+1}\wedge z=w_i)\vee(u_i>u_{i+1}\wedge z=v_i),\\
		0, & z\not\in\{v_i,w_i\}.
		\end{cases}
		\end{align}
		Since the right-hand side of \cref{vf(e_u)_vv=vf(e_u)_u_0u_0+sum} depends only on $\0_\vf$
		and $\vf(e_z)(u_0,u_0)$, we have $\vf(e_z)(v,v)=\psi(e_z)(v,v)$. This proves $\vf(e_z)=\psi(e_z)$, as $\vf(e_z),\psi(e_z)\in L_0$.
	\end{proof}
	
	\begin{definition}
		Let $\0:B\to B$ be a bijection and $\Gamma: u_0,u_1,\dots,u_m=u_0$ a closed walk in $X$. We introduce the following $4$ functions $X\to\mathbb{N}$:
		\begin{align*}
		s^+_{\0,\G}(z)&=|\{i: u_i<u_{i+1}\text{ and }\exists w>z\text{ such that }\0(e_{zw})=e_{u_iu_{i+1}}\}|,\\
		s^-_{\0,\G}(z)&=|\{i: u_i>u_{i+1}\text{ and }\exists w>z\text{ such that }\0(e_{zw})=e_{u_{i+1}u_i}\}|,\\
		t^+_{\0,\G}(z)&=|\{i: u_i<u_{i+1}\text{ and }\exists w<z\text{ such that }\0(e_{wz})=e_{u_iu_{i+1}}\}|,\\
		t^-_{\0,\G}(z)&=|\{i: u_i>u_{i+1}\text{ and }\exists w<z\text{ such that }\0(e_{wz})=e_{u_{i+1}u_i}\}|.
		\end{align*}
	\end{definition}
	
	\begin{remark}\label{0_vf-is-admissible}
		Let $\vf\in\wtl\laut(I(X,K))$ and $\0=\0_\vf$. Then for any closed walk $\Gamma:u_0,u_1,\dots,u_m=u_0$ in $X$ and for all $z\in X$
		\begin{align}\label{s^+-s^-=t^+-t^-}
		s^+_{\0,\G}(z)-s^-_{\0,\G}(z)=t^+_{\0,\G}(z)-t^-_{\0,\G}(z).
		\end{align}
		Indeed, by \cref{vf(e_z)_uu-vf(e_z)_vv=pm-one-or-zero} we have
		$
		0=\sum_{i=0}^{m-1}(\vf(e_z)(u_{i+1},u_{i+1})-\vf(e_z)(u_i,u_i))
		=-s^+_{\0,\G}(z)+s^-_{\0,\G}(z)+t^+_{\0,\G}(z)-t^-_{\0,\G}(z).
		$
	\end{remark}
	
	\begin{definition}\label{defn-admissible}
		A bijection $\0:B\to B$ is said to be \textit{admissible} if \cref{s^+-s^-=t^+-t^-} holds for any closed walk $\G:u_0,u_1,\dots,u_m=u_0$ in $X$ and for all $z\in X$. In particular, if $X$ is a tree, then any bijection $\0:B\to B$ is admissible.
	\end{definition}
	
	\begin{example}
		Let $X$ be a $2$-crown, i.e., $X=\{1,2,3,4\}$ with the following Hasse diagram.
		
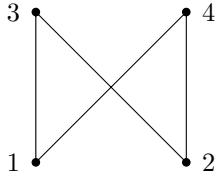
\begin{center}
			\begin{tikzpicture}
			\draw [fill=black] (-.3,0) node{$1$}   (-.3,2) node{$3$}  (2.3,0) node{$2$}  (2.3,2) node{$4$}
			(1.02,-0.5) (0,0) circle(0.05)--(0,1)--(0,2)circle(0.05)--(1,1)--(2,0)circle(0.05)--(2,1)--(2,2)circle(0.05) --(1,1)--(0,0);
			\end{tikzpicture}
		\end{center}
		Consider $\0:B\to B$ such that $\0(e_{13})=e_{14}$, $\0(e_{14})=e_{13}$, $\0(e_{23})=e_{23}$ and $\0(e_{24})=e_{24}$. Then $\0$ is not admissible. Indeed, for the cycle $\G: 1<3>2<4>1$ we have $s^+_{\0,\G}(3)=s^-_{\0,\G}(3)=t^+_{\0,\G}(3)=0$ and $t^-_{\0,\G}(3)=2$, so that $s^+_{\0,\G}(3)-s^-_{\0,\G}(3)\ne t^+_{\0,\G}(3)-t^-_{\0,\G}(3)$. Observe, however, that $s^+_{\0,\G}(1)=s^-_{\0,\G}(1)=1$ and $t^+_{\0,\G}(1)=t^-_{\0,\G}(1)=0$, so \cref{s^+-s^-=t^+-t^-} is not invariant under the choice of $z\in X$.
	\end{example}
	
	\begin{lemma}
		In \cref{defn-admissible} it suffices to verify \cref{s^+-s^-=t^+-t^-} for any cycle $\G$ in $X$.
	\end{lemma}
	\begin{proof}
		Let us fix a bijection $\0:B\to B$ and $z\in X$. Assume that \cref{s^+-s^-=t^+-t^-} is valid for any cycle. We will prove that it also holds for any closed walk $\G:u_0,u_1,\dots,u_m=u_0$. We call a pair of indices $(i,j)$, $i<j-1$,  a repetition if $u_i=u_j$. The proof will be by induction on the number of repetitions in $\G$. If $\G$ has only one repetition $(0,m)$ and $m\ge 4$, then it is a cycle, so \cref{s^+-s^-=t^+-t^-} holds by assumption. The case $m=3$ is impossible, since the Hasse diagram of a poset has no triangles. If $m=2$, then $\G$ is of the form $x<y>x$ or $x>y<x$. In this case it always satisfies \cref{s^+-s^-=t^+-t^-}. Indeed, let $\G:x<y>x$. If $\0(e_{zw})=e_{xy}$ for some $w>z$, then $s^{\pm}_{\0,\G}(z)=1$ and $t^{\pm}_{\0,\G}(z)=0$; and if $\0(e_{wz})=e_{xy}$ for some $w<z$, then $s^{\pm}_{\0,\G}(z)=0$ and $t^{\pm}_{\0,\G}(z)=1$; otherwise $s^{\pm}_{\0,\G}(z)=t^{\pm}_{\0,\G}(z)=0$. The case $\G:x>y<x$ is analogous.
		
		Now suppose that $\G$ has a repetition different from $(0,m)$, say $(i,j)$, $i<j-1$. If $i>0$ and $j<m$, then $\G$ splits into two closed walks $\G':u_0,\dots,u_i,u_{j+1},\dots,u_m$ and $\G'':u_i,\dots,u_j$ whose numbers of repetitions are strictly less than that of $\Gamma$. Hence, by induction hypothesis, $\G'$ and $\G''$ satisfy \cref{s^+-s^-=t^+-t^-}. Since, for any $0\le i<m$, either $u_i,u_{i+1}\in\G'$ or $u_i,u_{i+1}\in\G''$, then $s^{\pm}_{\0,\G}(z)=s^{\pm}_{\0,\G'}(z)+s^{\pm}_{\0,\G''}(z)$ and $t^{\pm}_{\0,\G}(z)=t^{\pm}_{\0,\G'}(z)+t^{\pm}_{\0,\G''}(z)$, so \cref{s^+-s^-=t^+-t^-} for $\G$ follows from \cref{s^+-s^-=t^+-t^-} for $\G'$ and $\G''$. If $i=0$, then $j<m-1$, and we apply the same argument to $\G':u_0,\dots,u_j$ and $\G'':u_j,\dots,u_m$. The case $j=m$ and $i>1$ is similar.
	\end{proof}
	
	\begin{remark}\label{sum-vf(e_z)(uu)-nonzero}
		Let $\vf\in\wtl\laut(I(X,K))$ and $u,v\in X$. Then $\sum_{z\in X}\vf(e_z)(u,u)=\sum_{z\in X}\vf(e_z)(v,v)\ne 0$.
		
		Indeed, observe that $\sum_{z\in X}e_z=\delta$, so that $\sum_{z\in X}\vf(e_z)=\vf(\delta)$. Since $\dl$ is central, so is $\vf(\delta)$, and it is obviously non-zero, therefore $\vf(\dl)=k\dl$ for some $k\in K^*$, by \cite[Corollary 1.3.15]{SpDo}.
	\end{remark}
	
	\begin{definition}
		Let $\0:B\to B$ be an admissible bijection and $\sg:X^2_<\to K^*$ compatible with $\0$. We say that $\vf\in\wtl\laut(I(X,K))$ \textit{induces} the pair $(\0,\sg)$ if \cref{exypermutation} holds for all $x<y$.
	\end{definition}
	
	\begin{lemma}\label{existense-of-tau_sg_0_c}
		Let $\0:B\to B$ be an admissible bijection which is monotone on maximal chains in $X$ and $\sg:X^2_<\to K^*$ compatible with $\0$. Then there exists $\vf\in\wtl\laut(I(X,K))$ inducing $(\0,\sg)$.	
	\end{lemma}
	\begin{proof}
		We first show that equality \cref{exypermutation} defines a Lie automorphism of $J(I(X,K))$. Let $x,y,z,w\in X$ such that $x<y$ and $z<w$. If $x\neq w$ and $y\neq z$, then $\vf([e_{xy},e_{zw}])=\vf(0)=0$. Moreover, if $\0(e_{xy})=e_{uv}$ and $\0(e_{zw})=e_{st}$, then $u\neq t$ and $v\neq s$ by \cref{0-inv-monotone}. Thus
		$$[\vf(e_{xy}),\vf(e_{zw})] = \sigma(x,y)\sigma(z,w)[\0(e_{xy}),\0(e_{zw})]=\sigma(x,y)\sigma(z,w)[e_{uv},e_{st}]=0.$$
		If $y=z$, then $\vf([e_{xy},e_{yw}])=\vf(e_{xw})$. On the other hand,
		\begin{align*}
		[\vf(e_{xy}),\vf(e_{yw})] & = \sigma(x,y)\sigma(y,w)[\0(e_{xy}),\0(e_{yw})]\\
		& =\sigma(x,y)\sigma(y,w)(\0(e_{xy})\0(e_{yw})-\0(e_{yw})\0(e_{xy})).
		\end{align*}
		If $\0$ is increasing on a maximal chain containing $x<y<w$, then
		$$
		[\vf(e_{xy}),\vf(e_{yw})] =\sigma(x,w)\0(e_{xw})=\vf(e_{xw}).
		$$
		And if $\0$ is decreasing on a maximal chain containing $x<y<w$, then
		$$[\vf(e_{xy}),\vf(e_{yw})] =-\sigma(x,w)(-\0(e_{xw}))=\vf(e_{xw}).$$
		Analogously, if $x=w$ we have $\vf([e_{xy},e_{zx}])=-\vf(e_{zy})=[\vf(e_{xy}),\vf(e_{zx})]$. Thus $\varphi$ is a Lie endomorphism of $J(I(X,K))$. Let $e_{uv}\in J(I(X,K))$. There exist unique $x<y$ in $X$ such that $\0(e_{xy})=e_{uv}$. Then $e_{uv}=\sigma(x,y)^{-1}\vf(e_{xy})=\vf(\sigma(x,y)^{-1}e_{xy})$ and so $\varphi$ is bijective.
		
		Now, we need to show that $\vf$ extends to a Lie automorphism of $I(X,K)$. Let us fix $u_0\in X$ and $\vf(e_z)(u_0,u_0)\in K$ ($z\in X$) in a way that $\sum_{z\in X}\vf(e_z)(u_0,u_0)\ne 0$. Given $v\in X$, there is a walk $\G: u_0,u_1,\dots,u_m=v$ from $u_0$ to $v$, since $X$ is connected. Then we define $\vf(e_z)(v,v)$ by formulas \cref{vf(e_u)_vv=vf(e_u)_u_0u_0+sum,vf(e_z)_u_(i+1)u_(i+1)-vf(e_z)_u_iu_i}. We will show that the definition does not depend on the choice of a walk from $u_0$ to $v$. For, if there is another walk $\G': u'_0=u_0,u'_1,\dots,u'_{m'}=v$ from $u_0$ to $v$, then there is a closed walk $\Omega: u_0,u_1,\dots,u_m=v=u'_{m'},u_{m+1}=u'_{m'-1},\dots,u_{m+m'}=u'_0=u_0$. Since $\0$ is admissible,
		\begin{align*}
		0&=-s^+_{\0,\Omega}(z)+s^-_{\0,\Omega}(z)+t^+_{\0,\Omega}(z)-t^-_{\0,\Omega}(z)\\
		&=\sum_{i=0}^{m+m'-1}\Delta_{\Omega,i}(z)=\sum_{i=0}^{m-1}\Delta_{\G,i}(z)+\sum_{i=m}^{m+m'-1}\Delta_{\Omega,i}(z),
		\end{align*}
		where $\Delta_{\Omega,i}(z)$ and $\Delta_{\G,i}(z)$ are given by \cref{vf(e_z)_u_(i+1)u_(i+1)-vf(e_z)_u_iu_i}. Hence
		\begin{align*}
		\sum_{i=0}^{m-1}\Delta_{\G,i}(z)&=-\sum_{i=m}^{m+m'-1}\Delta_{\Omega,i}(z)=\sum_{i=0}^{m'-1}(-\Delta_{\Omega,i+m}(z))\\
		&=\sum_{i=0}^{m'-1}\Delta_{\G',m'-i-1}(z)=\sum_{i=0}^{m'-1}\Delta_{\G',i}(z).
		\end{align*}
		So, $\vf(e_z)(v,v)$ is well-defined by \cref{vf(e_u)_vv=vf(e_u)_u_0u_0+sum}.
		
		Observe that, by the construction, $\vf$ satisfies \cref{vf(e_z)_uu-vf(e_z)_vv=pm-one-or-zero} for all $u<v$. Indeed, there is a walk starting at $u_0$ of the form $\G: u_0,\dots,u_l=u,\dots,u_m=v$, where $u_l<\dots<u_m$. Using this walk to define $\vf(e_z)(u,u)$ and $\vf(e_z)(v,v)$, by \cref{vf(e_u)_vv=vf(e_u)_u_0u_0+sum} we have
		\begin{align}\label{vf(e_z)(v_v)-vf(e_z)(u_u)=sum}
		\vf(e_z)(v,v)-\vf(e_z)(u,u)=\sum_{i=l}^{m-1}\Delta_{\G,i}(z).
		\end{align}
		Assume that $\0\m$ is increasing on a maximal chain containing $u_{l}<\dots<u_m$. Then there exist $v_{l}<\dots<v_m$ such that $\0(e_{v_iv_j})=e_{u_iu_j}$, $l\le i<j\le m$. If $z=v_{l}$, then $\Delta_{\G,l}(z)=-1$ and $\Delta_{\G,i}(z)=0$ for all $l<i\le m-1$ by \cref{vf(e_z)_u_(i+1)u_(i+1)-vf(e_z)_u_iu_i}. Hence, the right-hand side of \cref{vf(e_z)(v_v)-vf(e_z)(u_u)=sum} equals $-1$. But $\0(e_{zv_m})=e_{u_lu_m}=e_{uv}$, so \cref{vf(e_z)_uu-vf(e_z)_vv=pm-one-or-zero} holds for $u<v$. Similarly, if $z=v_m$, then $\Delta_{\G,m-1}(z)=1$ and $\Delta_{\G,i}(z)=0$ for all $l\le i<m-1$ by \cref{vf(e_z)_u_(i+1)u_(i+1)-vf(e_z)_u_iu_i}. Hence, the right-hand side of \cref{vf(e_z)(v_v)-vf(e_z)(u_u)=sum} equals $1$ and $\0(e_{v_lz})=e_{uv}$ justifying \cref{vf(e_z)_uu-vf(e_z)_vv=pm-one-or-zero} for $u<v$. If $z=v_k$ with $l<k<m$, then $\Delta_{\G,k}(z)=-1$, $\Delta_{\G,k-1}(z)=1$ and $\Delta_{\G,i}(z)=0$ for all $l\le i\le m-1$, $i\not\in\{k-1, k\}$. Hence, the right-hand side of \cref{vf(e_z)(v_v)-vf(e_z)(u_u)=sum} equals $0$ and $\0(e_{v_lv_m})=e_{uv}$ with $z\not\in\{v_l,v_m\}$, so \cref{vf(e_z)_uu-vf(e_z)_vv=pm-one-or-zero} still holds. Finally, if $z\ne v_i$ for all $l\le i\le m$, then $\Delta_{\G,i}(z)=0$ for all $l\le i\le m-1$, so that the right-hand side of \cref{vf(e_z)(v_v)-vf(e_z)(u_u)=sum} again equals $0$. The decreasing case is similar.
		
		Now, define $\vf(e_z)=\sum_{v\in X}\vf(e_z)(v,v)e_v\in D(X,K)$. Automatically, we have $[\vf(e_x),\vf(e_y)]=0$ for all $x,y$. Let $x<y$ and $\0(e_{xy})=e_{uv}$. Then it follows from \cref{vf(e_z)_uu-vf(e_z)_vv=pm-one-or-zero} that for all $z\in X$,
		\begin{align*}
		[\vf(e_z),\vf(e_{xy})]&=[\vf(e_z)_D,\sg(x,y)e_{uv}]=\sg(x,y)(\vf(e_z)(u,u)-\vf(e_z)(v,v))e_{uv}\\
		&=
		\begin{cases}
		\sg(x,y)e_{uv},  & z=x,\\
		-\sg(x,y)e_{uv}, & z=y,\\
		0,            & z\not\in\{x,y\}
		\end{cases}
		=
		\begin{cases}
		\vf(e_{xy}),  & z=x,\\
		-\vf(e_{xy}), & z=y,\\
		0,            & z\not\in\{x,y\}
		\end{cases}\\
		&=\vf([e_z,e_{xy}]).
		\end{align*}
		Thus, the extended $\vf$ is a Lie endomorphism of $I(X,K)$.
		
		It remains to prove that $\vf$ is bijective. Since $\vf$ maps bijectively $J(I(X,K))$ onto itself and $\vf(D(X,K))\subseteq D(X,K)$
		by construction, then it suffices to show that $\vf$ is bijective on $D(X,K)$. Suppose that $\sum_{z\in X}k_z\vf(e_z)=0$ for some
		$\{k_z\}_{z\in X}\sst K$, at least one of which is non-zero. Let $x<y$ and $\0(e_{xy})=e_{uv}$. Then
		\begin{align}
		\sum_{z\in X}k_z(\vf(e_z)(v,v)-\vf(e_z)(u,u)) & =\left(\sum_{z\in X}k_z\vf(e_z)\right)(v,v)-\left(\sum_{z\in X}k_z\vf(e_z)\right)(u,u)\notag\\
		& =0. \label{sum-k_z(vf(e_z)_vv-vf(e_z)_uu}
		\end{align}
		On the other hand, by \cref{vf(e_z)_uu-vf(e_z)_vv=pm-one-or-zero} the left-hand side of \cref{sum-k_z(vf(e_z)_vv-vf(e_z)_uu} equals $k_y-k_x$. Thus, $k_x=k_y$. Since $X$ is connected, we have $k_x=k_y\ne 0$ for all $x,y\in X$, whence
		\begin{align*}
		0=\sum_{z\in X}k_z\vf(e_z)(u_0,u_0)=k_{u_0}\sum_{z\in X}\vf(e_z)(u_0,u_0),
		\end{align*}
		which contradicts the choice of $\vf(e_z)(u_0,u_0)$ at the beginning of the proof. Therefore $\vf$ is a Lie automorphism of $I(X,K)$. Moreover, since $\vf(L_i)\subseteq L_i$ for all $i\geq 0$, then $\vf\in\wtl\laut(I(X,K))$, by \cref{rem_elementary}.
	\end{proof}

	\begin{definition}\label{defn-of-tau_sg_0_c}
		Let $X=\{x_1,\dots,x_n\}$. Given an admissible bijection $\0:B\to B$ which is monotone on maximal chains in $X$, a map $\sg:X_<^2\to K^*$ compatible with $\0$
		and a sequence $c=(c_1,\dots,c_n)\in K^n$ such that $\sum_{i=1}^nc_i\in K^*$, define $\tau=\tau_{\0,\sg,c}$ to be the elementary Lie
		automorphism of $I(X,K)$ such that $\tau|_{J(I(X,K))}$ is given by the right-hand side of \cref{exypermutation} and $\tau|_{D(X,K)}$ is determined
		by 
		\begin{align}\label{tau(e_x_i)(x_1_x_1)=c_i}
		\tau(e_{x_i})(x_1,x_1)=c_i,
		\end{align}
		$i=1,\dots,n$, as in \cref{0_vf=0_psi-and-vf(e_z)_uu=psi(e_z)_uu,existense-of-tau_sg_0_c}.
	\end{definition}
	
	\begin{theorem}\label{vf-decomp-as-tau_0_sg_c}
		Each $\vf\in\wtl\laut(I(X,K))$ can be uniquely represented in the form
		\begin{align}\label{vf=tau_0_sg_c}
		\vf=\tau_{\0,\sg,c},
		\end{align}
		where $\0:B\to B$ is an admissible bijection which is monotone on maximal chains in $X$, $\sg:X_<^2\to K^*$ is a map compatible with $\0$ and $c=(c_1,\dots,c_n)\in K^n$ is such that $\sum_{i=1}^nc_i\in K^*$.
	\end{theorem}
	\begin{proof}
		Let $X=\{x_1,\dots,x_n\}$. We define $\0=\0_\vf$, $\sg=\sg_\vf$ and $c_i=\vf(e_{x_i})(x_1,x_1)$, $i=1,\dots,n$. Then $\0$, $\vf$ and $c$ satisfy the desired properties by \cref{0_vf(e_xy)-and-0_vf(e_yz),0_vf-on-a-chain,0_vf-is-admissible,sum-vf(e_z)(uu)-nonzero}.
		Hence, $\tau:=\tau_{\0_\vf,\sg_\vf,c}$ is well-defined. For all $x_i<x_j$ we have
		$
		\vf(e_{x_ix_j})=\sg(x_i,x_j)\0(e_{x_ix_j})=\tau(e_{x_ix_j}),
		$
		so $\vf|_{J(I(X,K))}=\tau|_{J(I(X,K))}$. Moreover, $\vf|_{D(X,K)}=\tau|_{D(X,K)}$ by  \cref{0_vf=0_psi-and-vf(e_z)_uu=psi(e_z)_uu}. The uniqueness of \cref{vf=tau_0_sg_c} follows from \cref{exypermutation,tau(e_x_i)(x_1_x_1)=c_i}.
	\end{proof}
	
	\begin{corollary}\label{Laut-T_n(K)}
		Each Lie automorphism $\vf$ of $T_n(K)$ is a composition $\psi\circ\mu$, where $\psi$ is an automorphism or the negative of an anti-automorphism and $\mu$ is a Lie automorphism given by $\mu(e_{ij})=e_{ij}$ for all $i<j$ and $\mu(e_i)=e_i+\af_i\dl$ for all $i$, where $1+\sum_{i=1}^n \af_i\in K^*$. Equivalently, $\vf=\psi+\nu$, where $\nu(e_{ij})=0$ for all $i<j$ and $\nu(e_i)=\af_i\dl$ for all $i$, therefore $\vf$ is proper.
	\end{corollary}
	\begin{proof}
		Indeed, $T_n(K)\cong I(X,K)$ where $X$ is the chain $x_1=1<2<\dots<n=x_n$. Let $\vf\in\laut(I(X,K))$. In view of \cref{LAut-cong-Inn_1-rtimes-wtl-LAut} we may assume that $\vf\in\wtl\laut(I(X,K))$. Then $\vf=\tau_{\0,\sg,c}$ as in \cref{vf-decomp-as-tau_0_sg_c}.  The bijection $\0:B\to B$ is increasing on $X$ if and only if $\0=\id_B$. In this case $\vf(e_{ij})=\sg(i,j)e_{ij}$, $i<j$, for some $\sg:X^2_<\to K^*$ satisfying $\sg(i,k)=\sg(i,j)\sg(j,k)$, $i<j<k$. Observe that $\sg(i,j)=\sg(1,i)\m\sg(1,j)$, $1<i<j$, so $\vf(e_{ij})=\eta\m e_{ij}\eta$, $i<j$, where $\eta=e_1+\sum_{i=2}^n\sg(1,i)e_i$. Moreover, $\eta\m e_i\eta=e_i$ for any $i$. It also follows from \cref{vf(e_u)_vv=vf(e_u)_u_0u_0+sum,vf(e_z)_u_(i+1)u_(i+1)-vf(e_z)_u_iu_i} that $\vf(e_i)(j,j)=c_i+\dl_{ij}$ for all $i\ne 1$ and $\vf(e_1)(j,j)=c_1-1+\dl_{1j}$, whence $\vf(e_i)=e_i+\af_i\dl$, where $\af_1=c_1-1$ and $\af_i=c_i$, $i>1$. The decompositions of $\vf$ as $\psi\circ\mu$ and $\psi+\nu$, in which $\psi$ is the conjugation by $\eta\m$, are now clear.
		
		If $\0$ is decreasing on $X$, then $\0(e_{ij})=e_{n-j+1,n-i+1}$, $1\le i<j\le n$, so $\vf(e_{ij})=\sg(i,j)e_{n-j+1,n-i+1}$, where $\sg:X^2_<\to K^*$ is such that $\sg(i,k)=-\sg(i,j)\sg(j,k)$, $i<j<k$. Therefore, $\vf=\lb\circ\vf'$, where $\lb(e_{ij})=-e_{n-j+1,n-i+1}$ is the negative of an anti-automorphism of $I(X,K)$ and $\vf'$ is as in the increasing case.
	\end{proof}
	
	\begin{example}
		Let $X=\{1,2,3,4\}$, where $1<2<3$, $1<4$ and $4$ is incomparable with $2$ and $3$. Observe that $I(X,K)$ has no anti-automorphism  by \cite[Theorem 3]{BFS12}. Consider the following linear map:
		$
		\vf(e_1) = -e_3 - e_4,
		\vf(e_{12}) = e_{23},
		\vf(e_{13}) = -e_{13},
		\vf(e_{14}) = e_{14},
		\vf(e_2) = e_1 + e_3 + e_4,
		\vf(e_{23}) = e_{12},
		\vf(e_3) = e_2 + e_3,
		\vf(e_4) = e_4.    
		$
		Then $\vf=\tau_{\0,\sg,c}$, where $\0(e_{12})=e_{23},\0(e_{23})=e_{12},\0(e_{13})=e_{13},\0(e_{14})=e_{14}$, $\sg(1,2)=\sg(2,3)=\sg(1,4)=1,\sg(1,3)=-1$, $c_1=c_3=c_4=0,c_2=1$. Assume that $\vf$ is proper, i.e., $\vf=\psi+\nu$, where $\psi\in\Aut(I(X,K))$ and $\nu$ is a central-valued map. There is $i\in X$ such that $\psi(e_1)_D=e_i$ by \cite[Lemma 1]{Kh-aut}. Then $\vf(e_1)_D=e_i+\af\dl$ for some $\af\in K$, i.e., $\vf(e_1)(i,i)=\vf(e_1)(j,j)+1$ for all $j\ne i$, a contradiction.
	\end{example}
	\section*{Acknowledgements}
	This work was partially  supported by CNPq 404649/2018-1 and by the Funda\c{c}\~ao para a Ci\^encia e a Tecnologia (Portuguese Foundation for Science and Technology) through the project PTDC/MAT-PUR/31174/2017. We are grateful to the referee for pointing out a gap in the proof of \cref{elementary_phi_tilde}, as well as numerous small inaccuracies throughout the text. In particular, we thank the referee for the suggestion to reduce the proof of \cref{<e_xy><e_uv>-is-zero}.

	\bibliography{bibl}{}
	\bibliographystyle{acm}

\end{document}